\documentclass[reqno]{amsart}
\usepackage{amsfonts}
\usepackage{mathrsfs}
\usepackage{hyperref}
\usepackage{amssymb}
\usepackage{CJK}
 \newtheorem{thm}{Theorem}[section]
 \newtheorem{cor}[thm]{Corollary}
 \newtheorem{prop}[thm]{Corollary}
  \newtheorem{prob}[thm]{Problem}
 
 \newtheorem{conj}[thm]{Conjecture}
 \newtheorem{lem}[thm]{Lemma}

 \theoremstyle{remark}
 
 \numberwithin{equation}{section}
\DeclareMathOperator{\sign}{sign}

\newcommand{\de}{\textup{deg}}

\newcommand{\f}{\mathbb{F}_q}

\begin{document}
\title{Distance Distribution in Reed-Solomon Codes}

\author{Jiyou Li}
\address{Department of Mathematics, Shanghai Jiao Tong University, Shanghai, P.R. China}
\email{lijiyou@sjtu.edu.cn}

\author{Daqing Wan}
\address{Department of Mathematics, University of California, Irvine, CA 92697-3875, USA}
\email{dwan@math.uci.edu}
%



\begin{abstract}
Let $\f$ be the finite field of $q$ elements.
In this paper we obtain bounds on the following counting problem: given a polynomial $f(x)\in \f[x]$ of degree $k+m$ and a non-negative integer $r$, count the number of polynomials $g(x)\in \f[x]$ of degree at most $k-1$ such that $f(x)+g(x)$ has exactly $r$ roots in $\f$.
Previously, explicit formulas were known only for the cases $m=0, 1, 2$.
As an application, we obtain an asymptotic formula  on the list size of  the standard Reed-Solomon code $[q, k, q-k+1]_q$.
%

%
%
%

\end{abstract}

\maketitle \numberwithin{equation}{section}
\newtheorem{theorem}{Theorem}[section]
\newtheorem{lemma}[theorem]{Lemma}
\newtheorem{example}[theorem]{Example}
\allowdisplaybreaks

\section{Introduction}

\subsection{Motivations}

This paper is motivated by  the following fundamental coding theory problem:

\begin{prob}
 Let $\mathcal{C}$ be a linear code over $\f$. Given a received word  $u$,  determine the distance distribution having $u$ as the center.
That is, for integer $i\geq 0$, compute the number $N_i(u)$ of codewords in $\mathcal{C}$ whose distance to $u$ is exactly $i$.
\end{prob}

When the received word $u$ is a codeword, this is the classical weight
distribution problem, which is generally {\bf NP}-hard and only well understood for certain special codes such as MDS codes and some special families of cyclic codes. When the received word is not a codeword, it is equivalent to the coset weight distribution problem. The coset weight
distribution was determined for a few very special classes of linear codes including $t$-error-correcting BCH codes for $t\leq 3$  (cf. \cite{C1, CHZ, CZ}), external  self-dual binary codes of length $n$ for $n\leq 20, n=28, 40, 46, 56$ (cf. \cite{Ha, HN, O2, O1}) and the second-order {R}eed-{M}uller  code of length 64 (cf. \cite{BA, OW}).

The distance distribution problem can be viewed as the counting version of list decoding and is much harder and widely open even for standard Reed-Solomon codes.    In this paper, we make the first attempt to study this problem and obtain an asymptotic formula for standard Reed-Solomon codes.

A special case of our problem is computing the
error distance from a received word $u$, that is, finding the smallest non-negative integer $i$ such that $N_i(u)>0$.
This can be reduced to the decision version of the maximal likelihood decoding problem in coding theory.   As Reed Solomon codes are
constructed using polynomials, all such problems on Reed-Solomon codes can be reduced to polynomial factorization problems.
Details will be explained in Section 2. To be more precise in this introduction, we now introduce some notations.


Let $\f$ be the finite field of $q$ elements with characteristic $p$.
Let $1\leq n\leq q$ be a positive integer,
$D=\{x_1,\ldots, x_n\} \subset \f$ be a subset of
cardinality $|D|=n>0$. For $1\leq k\leq n$, the Reed-Solomon code
$\mathcal{RS}_{n,k}$ has the codewords of the form
$$(f(x_1), \ldots, f(x_n))\in \f^n,$$
where $f$ runs over all polynomials in $\f[x]$ of degree at
most $k-1$. It is well-known that the minimum distance of the Reed-Solomon code is $n-k+1$.
If $D=\f$ (or $\f^*)$, then the code $\mathcal{RS}_{q,k}$($\mathcal{RS}_{q-1,k}$) is called the standard (respectively the primitive)
Reed-Solomon codes. All our results for standard Reed-Solomon codes extend to
primitive Reed-Solomon codes with minor modification. For this reason, we shall focus on the standard
Reed-Solomon codes in this paper.

For  any word $u=(u_1,u_2,\ldots,u_{n})\in \f^n$,  one can efficiently
compute a unique polynomial $u(x)\in \f[x]$ of degree
at most $n-1$ such that
$$u(x_i)=u_i,  \ { \rm for ~ all} ~1\leq i\leq n.$$Explicitly, the
polynomial $u(x)$ is given by the Lagrange interpolation formula
$$u(x) = \sum_{i=1}^n u_i \frac{\prod_{j\not=i}(x-x_j)}{\prod_{j\not=i}(x_i-x_j)}.$$
The degree  $\de(u)$ of $u$ is then defined as the degree of the associated polynomial $u(x)$. It is easy to see that $u$ is a codeword if and only if $\de(u)<k.$

For a given word $u\in \f^n$,  the distance from $u$ to $\mathcal{RS}_{n,k}$ is defined by
$$d(u, \mathcal{RS}_{n,k}): = \min_{v\in \mathcal{RS}_{n,k}}d(u,v).$$
The maximum likelihood decoding of $u$ is to find a codeword $v\in
\mathcal{RS}_{n,k}$ such that $d(u,v) = d(u, \mathcal{RS}_{n,k})$. Thus, computing
$d(u, \mathcal{RS}_{n,k})$ is essentially the decision version for the maximum likelihood decoding problem, which is ${\bf NP}$-complete for
general subset $D\subset \f$, see Guruswami-Vardy \cite{GV} and Cheng-Murray \cite{CM}. For standard Reed-Solomon code
with $D=\f$, the complexity of the maximum
likelihood decoding is unknown to be {\bf NP}-complete. This is an
important open problem. It was shown by Cheng-Wan
\cite{CW1,CW2} that decoding the standard Reed-Solomon code is at least as hard as the discrete logarithm
problem in a large extension of the finite field $\f$.

If $\de(u)\leq k-1$,  then $u$ is a codeword and thus
$d(u, \mathcal{RS}_{n,k})=0$. We shall assume that $k \leq \de(u)\leq n-1$. The
following simple result gives an elementary bound for $d(u,
\mathcal{RS}_{n,k})$.

\begin{thm}\cite{LW1} Let $u\in {\bf F}_q^n$ be a word such that $k \leq \de(u)\leq n-1$.
Then,
$$n-\de(u) \leq d(u, \mathcal{RS}_{n,k}) \leq n-k.$$
\end{thm}

The word $u$ is called a deep hole if $d(u, \mathcal{RS}_{n,k})=n-k$, that is, it
achieves the covering radius. When $\de(u)=k$, the upper bound and the lower bound agree and hence $u$ is a deep hole. This
gives $(q-1)q^k$ deep holes. For a general Reed-Solomon code
$\mathcal{RS}_{n,k}$, it is already difficult to determine if a given word $u$ is a deep hole. Even for the special case that $\de(u)=k+1$, the deep hole problem is equivalent to the $(k+1)$-subset sum problem over $\f$
which is {\bf NP}-complete \cite{CM}.

For the standard Reed-Solomon code, that is, $D=\f$ and
thus $n=q$, there is the following deep hole conjecture of
Cheng-Murray \cite{CM}.
 \begin{conj}
For the code $\mathcal{RS}_{q, k}$ with $p>2$,  the set $\{u\in \f^n \big| \de(u)=k\}$
gives the set of all deep holes.
\end{conj}
 Many results were proved towards this conjecture. Please refer to \cite{CMP12}, \cite{Kai}, \cite{Liao}, \cite{ZCL} and the references there.

The deep hole problem is to determine when the upper bound in the
above theorem agrees with $d(u, \mathcal{RS}_{n,k})$. One is also interested in the situations
when the lower bound $n-\de(u)$ agrees with $d(u, \mathcal{RS}_{n,k})$. 
We call
$u$ {\bf ordinary} if $d(u, \mathcal{RS}_{n,k})=n-\de(u)$. A basic problem is
then to determine when a given word $u$ is ordinary. This is equivalent to determining if $N_{n-\de(u)}(u)>0$.
This problem will be studied in a future paper.
%

Since $k \leq \de(u) \leq n-1$, we can write $\de(u) = k+m$ for some non-negative integer $m \leq n-k-1$.
Then, the word $u$ is represented uniquely by
a polynomial $u(x)\in \f[x]$ of degree $k+m$.
For $0\leq r\leq k+m$, let $N_D(f(x), r)$ denote the number of  polynomials $g(x)\in\f[x]$  with  $\deg g(x) \leq k-1$ such that
$ f(x)+g(x)$ has exactly $r$ distinct roots in $D$.
It is clear that $N_i(u) = N_D(u(x), n-i)$. Thus, it is enough to study $N_D(f(x), r)$.

From now on, we only work with the standard Reed-Solomon codes $\mathcal{RS}_{q,k}$. Since $D=\f$, we can write $N(f(x), r) = N_{\f}(f(x), r)$
and $N_i(u) = N(u(x), q-i)$. It is clear that without loss of generality, we can assume that $f(x)$ is monic with no terms of degree less than $k$.
Our distance distribution problem for the standard Reed-Solomon code is reduced to the following number theoretic problem.

\begin{prob} Let $1\leq k \leq q$ and $-k \leq m\leq q-k-1$.
Given a monic polynomial $f(x)\in \f [x]$ of degree $k+m$ and an integer $0\leq r\leq k+m$, count $N(f(x), r)$,  the number  of polynomials $g(x)\in\f[x]$  with  $\deg g(x) \leq k-1$ such that $f(x)+g(x)$ has exactly $r$ distinct roots in $\f$.
\end{prob}
%
%

Not much is  known about this problem. Elementary explicit formulas for $m\leq 2$ were known before.
Exponential lower bounds and asymptotic formula for $N(f(x), r)$ have been studied in \cite{CW0}, \cite{CW1}, \cite{CW2}, \cite{LW2} in the extreme case $r=m+k$.
Our contribution of this paper is to prove results for all $0\leq r\leq k+m$.
If $k$ is very small (say logarithmic in $q$),  one can use the Chebotarev density theorem to derive a
good asymptotic formula. However, in coding theory application, $k$ is the code dimension which can be as large as a linear function of $q$. The problem then becomes more difficult.  The main purpose of this paper is to prove nontrivial results for large $k$ and a wide range of $r$ if $m$ is not too large.

\subsection{Known Cases for $m\leq 2$}

When $m<0$, $f(x)$ represents a codeword and thus we may assume $f\equiv 0$,  or equivalently $u=0$.  By a famous theorem of Mac Williams,
 for $0\leq r\leq k-1$ we have
    \begin{align*}
 N(0, r)
 ={q \choose r}q^{k-r-1}(q-1)\left(\sum_{j=0}^{k-r-1}(-1)^j {q-r-1 \choose j}q^{-j}\right).
    \end{align*}

If {\bf $m=0$}, then $\text{deg}(f)=k$. In this case, $u$ is a deep hole. An explicit formula for  $N(x^k, r)$ was given by A. Knopfmacher and J. Knopfmacher \cite{KK}.
%

 If {\bf $m=1$}, then $\text{deg}(f)=k+1$. We may assume $f(x)=x^{k+1}+ax^k$.  It turns out that $N(x^{k+1}+ax^k, r)$ depends on $a$. An explicit formula
 for $N(x^{k+1}+ax^k, r)$ was given by Zhou, Wang and Wang \cite{ZWW}.
A  more complicated explicit counting formula for the case {\bf $m=2$} is also given in the same paper.

When $m>2$, it is no longer reasonable to expect an explicit formula for $N(f(x), r)$,
but we can hope for an asymptotic formula. This is the aim of the present paper.

\subsection{Main Result}

For an integer $s\geq 0$, define the alternating sum
$$\mu_s=\displaystyle{\sum_{j=0}^{s} (-1)^{j}{q-r \choose j} q^{-j}} = 1 - \frac{q-r}{q} + {q-r\choose 2}\frac{1}{q^2}-\cdots.$$
The absolute value of the $j$-th term is decreasing in $j$. It follows that if $r =cq$ for some constant $0<c<1$, then
   $$0< c \leq \frac{r}{q} = 1 - \frac{q-r}{q} \leq \mu_{s} \leq 1.$$
    Since $\mu_s$ is a truncation of $(1-q^{-1})^{q-r}$, $\mu_{s}$ is close to $(1/e)^{1-c}$ when $q$ and $s$ are both large, and $r\approx cq$.
Our main result is the following bound on $N(f(x), r)$,  which holds for all $k, m$, and $r$.

\begin{thm}\label{Theorem1.5}
Let $f(x) \in \f [x]$ be a polynomial of $\text{deg}(f)=k+m\leq q-1$.
For all integers $0\leq r\leq k+m$, we have
\begin{align*}
\left|N(f(x), r)-\mu_{k+m-r}{q \choose r}q^{k-r}
\right|\leq \sum_{j=k+1}^{k+m}  {j \choose r}{\frac {q}p+m\sqrt{q}+j \choose j}{m-1 \choose k+m-j} \sqrt{q}^{k+m-j}.
 \end{align*}
\end{thm}
Our technique to establish Theorem \ref{Theorem1.5} is
based on a distinct coordinates sieving technique discovered by the authors \cite{LW2, LW3},  a weighted inclusion-exclusion sieving formula, and a character sum bound on constant degree polynomials defined over a suitable residue ring.

The number $\sqrt{q}$ in the error term comes from the application of the Riemann hypothesis over finite fields (Weil's bound).
The number of non-zero error terms in the error estimate is $k+m -\max \{k+1, r \}$. This means that if either $m$ is small
or $k+m-r$ is small, then there are only a few terms in the error estimate. The theorem also becomes stronger in the case $q=p$ is a prime since then the number $\frac {q}p $ becomes 1.
We now derive a few corollaries and explain how they are related to previous results.

 When  $m=0$, we may suppose $f(x)=x^k$. In this case, there is no error term in our asymptotic formula and  we thus obtain the following explicit formula
 first proved in \cite{KK}, as reported in the above known cases.
\begin{cor}
\begin{align*}
    N(x^k, r)={q \choose r}q^{k-r}\left(\sum_{j=0}^{k-r}(-1)^j {q-r \choose j}q^{-j}\right).
   \end{align*}
\end{cor}

 When  $r=k+m$,  there is only one term in the error estimate and we obtain the following corollary, which was first  proved
in \cite{LW2}.
\begin{cor} Let $f(x) \in \f [x]$ be a polynomial of $\text{deg}(f)=k+m\leq q-1$. Then,
\begin{align*}
\left|N({f(x), k+m})-\frac{1}{q^{m}}{q \choose k+m}\right|\leq
{\frac {q}p+m\sqrt{q}+k+m \choose k+m}.
 \end{align*}
\end{cor}

When $r=k+m-1$,  there are two terms in the error estimate. Combining the two terms, we obtain the following corollary,  which is already a new result.
 \begin{cor}Let $f(x) \in \f [x]$ be a polynomial of $\text{deg}(f)=k+m\leq q-1$. Then,
 \begin{align*}
&\left|N(f(x), k+m-1)-\frac{k+m-1}{q^{m}}{q \choose k+m-1}\right|\\
&\leq  {\frac {q}p+m\sqrt{q}+k+m \choose k+m}((m-1)\sqrt{q}+k+m).
 \end{align*}
 \end{cor}

For general $r$, there will be more terms in the error estimate. This makes it harder to estimate the error term.
However, as we shall see, the above $j$-th error term in the error estimate is sometimes increasing in $j$ and thus
we can combine all the error terms into a single error term. This helps in obtaining a much simpler asymptotic formula, as done
in next subsection.

The paper is organized as follows. In the end of this introductory section,  some asymptotic analysis for some special parameters are given.
In section 2, we prove the main result Theorem \ref{Theorem1.5}  by a key counting formula given in Lemma \ref{lem1.1}. In section 3 and 4, we introduce a sieving technique and a character sum derived by the Weil bound respectively.  The proof of Lemma \ref{lem1.1} will be given in Section 5.

\subsection{Asymptotic Analysis}
As an illustration, we show that our bound above can be used to give a nontrivial asymptotic formula.
We assume  $q=p$ is prime for simplicity.  Then we find simple conditions under which the error term can be significantly simplified.
Please note that the binomial
coefficients for real numbers are defined by $${a \choose b}=\frac{\Gamma{(a+1)}}{\Gamma{(b+1)}\Gamma{(a-b+1)}}.$$
\begin{cor} \label{Cor1.9}
 Let $q=p$ and $f(x)$ be a polynomial of degree $k+m$. Suppose $k=cp, m=p^{\delta}, r=k+p^\lambda$, where $c\in(0, 1), \delta\in(0, 1/4), \lambda\in(0, \delta)$ are constants. As $p$ goes to infinity,  we have
 $$N(f(x), r)=\mu_{k+m-r}{p \choose r}p^{k-r}(1+o(1)).$$
\end{cor}

\begin{proof}
By Theorem \ref{Theorem1.5}, we have
 \begin{align*}
\left|N(f(x), r)-\mu_{k+m-r} {p \choose r}p^{k-r}
\right|&\leq \sum_{j=r}^{k+m}  {j \choose r}{m\sqrt{p}+1+j \choose m\sqrt{p}+1}{m-1 \choose k+m-j} p^{\frac{k+m-j}2}\\
 &\leq m\cdot \max_{r\leq j\leq k+m} E_j,
 \end{align*}
where $$E_j={j \choose r}{m\sqrt{p}+1+j \choose m\sqrt{p}+1}{m-1 \choose k+m-j} {p}^{\frac {k+m-j}2}.$$
One computes that for $r\leq j<k+m$,
$$\frac {E_{j+1}}{E_j}=\frac{(j+1)}{(j+1-r)}\cdot \frac {(m\sqrt{p}+j+2)}{(j+1)} \cdot \frac{(k+m-j)}{(j-k)\sqrt{p}} .$$
       Write $j=r+j'$, where $0 \leq j'< k+m-r=p^{\delta}-p^{\lambda}$. Then
      $$\frac {E_{j+1}}{E_j}= \frac {(m\sqrt{p}+r+j'+2)}{(j'+1)}\cdot\frac{(p^{\delta}-p^{\lambda}-j')}{(p^{\lambda}+j')\sqrt{p}}.$$
Since $0\leq j' < p^{\delta}-p^{\lambda}$, we deduce
 $$\frac {E_{j+1}}{E_j}> \frac {(m\sqrt{p}+r)}{(p^{\delta}-p^{\lambda})}\cdot\frac{1}{p^{\delta}\sqrt{p}}.$$
Note that $r\geq cp$ and $\lambda <\delta <\frac{1}{4}$. It follows that for $p$ sufficiently large, we have $E_{j+1}/E_j >1$ for all $j$,
thus $E_j$ is increasing in $j$ and
$$\max_{r\leq j\leq k+m}E_j={k+m \choose r}{m\sqrt{p}+k+m+1 \choose m\sqrt{p}+1}.$$
    %
As noted in the beginning of this section,
   $$0< c \leq \frac{r}{p} = 1 - \frac{p-r}{p} \leq \mu_{k+m-r} \leq 1.$$
 To complete the proof of the corollary,  it suffices to show
  \[\lim_{p\rightarrow\infty}\frac {m {k+m \choose k+m-r}{m\sqrt{p}+1+k+m \choose m\sqrt{p}+1}}{{p \choose r}p^{k-r}}=0.\]
Since  $k+m-r \leq m \leq (k+m)/2$ and $1 \leq r-k\leq m$, it is enough to prove
 \[\lim_{p\rightarrow\infty}\frac {m {k+m \choose m}p^m{m\sqrt{p}+1+k+m \choose m\sqrt{p}+1}}{{p \choose r}}=0.\]
By the inequalities
 \[(\frac {n}l)^l \leq {n \choose l} \leq (\frac {en}l)^l,\]
it is sufficient to have
 \[\lim_{p\rightarrow\infty}\frac{m
 (e(k+m)p)^{m}(e+e\frac {k+m}{m\sqrt{p}+1})^{m\sqrt{p}+1}}{(\frac pr)^r}=0.\]
 Since $k=cp, m=p^{\delta}, r=cp+p^\lambda$ and $c\in(0, 1), \delta\in(0, 1/4), \lambda\in(0, \delta)$, by taking logarithm, it is equivalent to have
  \[\lim_{p\rightarrow\infty}\left({\delta\ln p+2p^\delta \ln p+ p^{1/2+\delta}\ln (e+2cp^{1/2-\delta})}+{cp \ln c}\right)=-\infty.\]
This is clearly satisfied since $0<c<1$. We obtain the desired
asymptotic formula
 $$N(f(x), r)=\mu_{k+m-r}{p \choose r}p^{k-r}(1+o(1)).$$
\end{proof}

Note that our asymptotic analysis here only considers the case $q=p$, $k$ is large, $m$ is small and $r$ is large.  It is certainly possible to find other range of parameters for which the same asymptotic formula holds. However, note that $m$ must be bounded by $\sqrt{q}$ in order for our estimate gives a non-trivial estimate. To keep this paper focused, such finer analysis together with its applications to list decoding and bounded distance decoding will be
discussed in a future work.

Beside coding theory, our result may have potential applications in number theory and graph theory. In number theory, it is a classical problem to
understand the factorization pattern of a family of polynomials.  In graph theory, it is related to the spectrum distribution of Wenger type graphs,
see \cite{CLWW}.

\section{ Proof of the main theorem}
In this section we prove the following main result (Theorem \ref{Theorem1.5} in Section 1).
\begin{thm} \label{Theorem3.1}
Let $f(x) \in \f [x]$ be a polynomial of $\text{deg}(f)=k+m\leq q-1$.
For all integers $0\leq r\leq k+m$, we have
\begin{align*}
\left|N(f(x), r)-\mu_{k+m-r}{q \choose r}q^{k-r}
\right|\leq \sum_{j=k+1}^{k+m}  {j \choose r}{\frac {q}p+m\sqrt{q}+j \choose j}{m-1 \choose k+m-j} \sqrt{q}^{k+m-j}.
 \end{align*}
 \end{thm}
 The main technique of the proof is a weighted sieving formula and
 the following counting lemma,  which will be proved in Section 5.

\begin{lem}\label{lem1.1} Let $f(x) \in \f [x]$ be a monic polynomial of degree $d=k+m\leq q-1$. Let $M(f, r)$  denote
the number of pairs $(D_r, g(x))$ with $D_r$ being a $r$-subset in $\f$ and $g(x)\in \f[x]$ of degree at most $k-1$ satisfying
$$ (f(x)+g(x))|_{D_r} \equiv 0.$$ Then for $k+1\leq r\leq d$,  we have
 \begin{align*}
\left| {M(f, r)}-{{q \choose r}q^{k-r}}\right| \leq  {\frac {q}p+m\sqrt{q}+r \choose r}{m-1 \choose d-r} \sqrt{q}^{d-r}.
 \end{align*}
\end{lem}

\begin{proof}[{\bf Proof of Theorem \ref{Theorem3.1}}]
The proof is based on two different kinds of inclusion-exclusion sievings.
We shall let $g(x) \in \f[x]$ denote a polynomial of degree at most $k-1$.
For $c\in \f$, let $P_c$ denote the
property that $f(x)+g(x)$ has $c$ as a root. For a subset $C\subseteq \f$,  let  $N_C$  be the number of $g(x)$ such that $f(x)+g(x)$
has property $P_c$ for each $c\in {C}$.  The $|C|\times |C|$ Vandermonde matrix formed using the elements of ${C}$ is non-singular.
It follows by linear algebra that for $|C|\leq k$, we have $N_C=q^{k-|C|}$.
In the case $r=0$,
the inclusion-exclusion sieving \cite{St} implies that
 \begin{align*}
N(f, 0)&=q^k-\sum_{c\in \f}N_{\{c\}}+
\cdots+(-1)^{d}\sum_{\{c_1, c_2, \dots, c_{d}\}\subset \f}N_{\{c_1, c_2, \dots, c_{d}\}}\\
&=q^k-{q \choose 1} q^{k-1}+{q \choose 2} q^{k-2}-\cdots+(-1)^k{q \choose k} q^{0}+\sum_{j=k+1}^d(-1)^j N_j,
 \end{align*}
where $N_j$ is the number of pairs $(D_j, g(x))$ with $D_j$ being a $j$-subset in $\f$ and $g(x)\in \f[x]$ of degree at most $k-1$ satisfying
$$ (f(x)+g(x))|_{D_j} \equiv 0.$$
Applying Lemma \ref{lem1.1}, we have
\begin{align*}
\left|N(f, 0)-\sum_{i=0}^{k+m} (-1)^i {q \choose i}q^{k-i}
\right|&\leq \sum_{j=k+1}^d{\frac {q}p+m\sqrt{q}+j \choose j}{m-1 \choose d-j} \sqrt{q}^{d-j}.
 \end{align*}
This proves the theorem in the case $r=0$. More generally,
for $0\leq r\leq d$, using the weighted inclusion-exclusion sieving formula,  we deduce
\begin{align*}
N(f, r)&=\sum_{\{c_1, c_2, \dots, c_{r}\}\subset \f}N_{\{c_1, c_2, \dots, c_{r}\}}-{r+1 \choose r} \sum_{\{c_1, c_2, \dots, c_{r+1}\}\subset \f}N_{\{c_1, c_2, \dots, c_{r+1}\}}+\cdots\\
 &=\sum_{j=r}^{k} (-1)^{j-r} {j \choose r}{q \choose j} q^{k-j}+\sum_{j=k+1}^d(-1)^{j-r} {j \choose r}N_j.
 \end{align*}
Applying Lemma \ref{lem1.1} again, we have
\begin{align*}
&\left|N(f, r)-\sum_{j=r}^{d} {j \choose r}{q \choose j}(-1)^{j-r} q^{k-j}
\right|\\
&\leq \sum_{j=k+1}^d{j \choose r}{\frac {q}p+m\sqrt{q}+j \choose j}{m-1 \choose d-j} \sqrt{q}^{d-j}.
  \end{align*}
By the elementary properties of binomials, the main term can be rewritten and we obtain the following final form
\begin{align*}
&\left|N(f, r)-{q \choose r}q^{k-r}\displaystyle{\left(\sum_{j=0}^{d-r} (-1)^{j}{q-r \choose j} q^{-j}\right)}
\right|\\
&\leq \sum_{j=k+1}^d{j \choose r}{\frac {q}p+m\sqrt{q}+j \choose j}{m-1 \choose d-j} \sqrt{q}^{d-j}.
 \end{align*}
The theorem is proved.
\end{proof}

\section{A distinct coordinate sieving formula}
In this section we introduce a sieving formula, which is a main technique for establishing Lemma \ref{lem1.1} and might have its own interests.
Roughly speaking, this formula significantly improves the classical
inclusion-exclusion sieve in many distinct coordinates counting
problems. We cite it here without proof. For details and related
applications please refer to \cite{LW2, LW3}.

Let $\Omega$ be a finite set, and let $\Omega^k$ be the Cartesian
product of $k$ copies of $\Omega$. Let $X$ be a
subset of $\Omega^k$. 
Define $\overline{X}=\{(x_1,x_2,\ldots,x_k)\in X \ | \ x_i\ne x_j,
\forall i\ne j\}.$
Let $f(x_1,x_2,\dots,x_k)$ be a complex valued function defined over
$X$ and
$$F=\sum_{x \in \overline{X}}f(x_1,x_2,\dots,x_k).\label{1.00}$$
Many problems arising in number theory and coding theory are reduced to evaluate $F$ very carefully. However, the direct inclusion-exclusion
sieving has too many terms and thus usually produces too much errors.
Roughly speaking, our formula describes what happens for those cancellations and make it possible to compute $F$ explicitly.

 Let $S_k$ be the symmetric group on $\{1,2,\ldots, k\}$.
Each permutation $\tau\in S_k$ factorizes uniquely 
as a product of disjoint cycles and  each
  fixed point is viewed as a trivial cycle of length $1$.
   Two permutations in $S_k$ are conjugate if and only if they
have the same type of cycle structure (up to the order).
 For $\tau\in S_k$, define the sign of $\tau$ to
  $\sign(\tau)=(-1)^{k-l(\tau)}$, where $l(\tau)$ is the number of
 cycles of $\tau$ including the trivial cycles.
 For a permutation $\tau=(i_1i_2\cdots i_{a_1})
  (j_1j_2\cdots j_{a_2})\cdots(l_1l_2\cdots l_{a_s})$
  with $1\leq a_i, 1 \leq i\leq s$, define
  \hskip 1.0cm
  \begin{align} \label{1.1}
     X_{\tau}=\left\{
(x_1,\dots,x_k)\in X,
 x_{i_1}=\cdots=x_{i_{a_1}},\ldots, x_{l_1}=\cdots=x_{l_{a_s}}
 \right\}.
\end{align}
 Similarly, for $\tau \in S_k$,  define $F_{\tau}=\sum_{x \in
X_{\tau} } f(x_1,x_2,\dots,x_k). $ Now we can state our sieve
formula.   We remark that there are many other interesting
corollaries of this formula. For interested reader we refer to
\cite{LW2}.

\begin{thm} \label{thm1.0}Let $F$ and $F_{\tau}$ be defined as above. Then   \begin{align}
  \label{1.5} F=\sum_{\tau\in S_k}{\sign(\tau)F_{\tau}}.
    \end{align}
 \end{thm}

Note that the symmetric group $S_k$ acts on $\Omega^k$ naturally by
permuting coordinates. That is, for $\tau\in S_k$ and
$x=(x_1,x_2,\dots,x_k)\in \Omega^k$,  $\tau\circ
x=(x_{\tau(1)},x_{\tau(2)},\dots,x_{\tau(k)}).$
  A subset $X$ in $\Omega^k$ is said to be symmetric if for any $x\in X$ and
any $\tau\in S_k$, $\tau\circ x \in X $.
%
%
%
 For $\tau\in S_k$, denote by $\overline{\tau}$
 the conjugacy class determined by $\tau$ and it can
also be viewed as the set of permutations conjugated to $\tau$.
Conversely, for given conjugacy class $\overline{\tau}\in C_k$,
denote by $\tau$ a representative permutation of this class. For
convenience we usually identify these two symbols.

In particular,  if  $X$ is symmetric and $f$ is a symmetric function
under the action of $S_k$,  we then have the following simpler
formula than (\ref{1.5}).
\begin{prop} \label{thm1.1} Let $C_k$ be the set of conjugacy  classes
 of $S_k$.  If $X$ is symmetric and $f$ is symmetric, then
 \begin{align}\label{7} F=\sum_{\tau \in C_k}\sign(\tau) C(\tau)F_{\tau},
  \end{align} where $C(\tau)$ is the number of permutations conjugated to
  $\tau$.
\end{prop}

\section{Bounds on Character Sums}

 Let $f(x)\in \f[x]$ be a monic polynomial of degree $n>0$.
 Let $\chi$ be a group homomorphism from $(\f[x]/(f(x)))^*$ to $\mathbb{C}^*$.
 We extend this definition to $\f[x]/(f(x))$ by defining $\chi(g)=0$ for $\gcd(g, f)\ne 1$.
 Define $$M_k(\chi)=\sum_{g\in\f[x], \text{monic}, \deg(g)=k} \chi(g).$$

 \begin{lem}\label{weil} Assume that $\chi$ is non-trivial.  Then for $k \geq 0$,
 $$|M_k(\chi)|\leq {n-1 \choose k} {\sqrt{q}}^k.$$
 Furthermore, if $\chi(\f^*)=1$, then for $n\geq 2$, we have
 $$\left| \sum_{g\in\f[x], \text{monic}, \deg(g)\leq k} \chi(g)\right|\leq {n-2 \choose k} {\sqrt{q}}^k.$$
 \end{lem}

 \begin{proof}
The Dirichlet L-function of $\chi$ is
\begin{align*}
 L(\chi, t)&=\sum_{g\in\f[x], \text{monic}} \chi(g)t^{\deg(g)}\\
 &= \sum_{k=0}^{\infty} M_k(\chi) t^k\in 1+t\mathbb{C}[[t]].
 \end{align*}
  If $k\geq n$, for monic $g$ of degree $k$, we can write uniquely $g=g_1 f+h$, where $g_1$ is monic in $\f[x]$, $\deg(g_1)= k-n$ and $h\in \f[x], \deg(h)\leq n-1$.
  Thus in this case,
  \begin{align*}
M_k(\chi)&=\sum_{g_1\in\f[x], \text{monic}, \deg(g_1)=k-n}~~\sum_{\deg(h)\leq n-1} \chi(h)\\
 &=q^{k-n}\sum_{h\in \f[x]/(f(x))}\chi(h)\\
 &=0.
 \end{align*}
 This implies
 \begin{align*}
 L(\chi, t)&=\prod_{i=1}^r(1-\rho_i t)
 \end{align*}
 is a polynomial of degree $\leq n-1$, i.e., $r\leq n-1$.
 By the Weil bound (\cite{Wan} Theorem 2.1),
 $$|\rho_i|\leq \sqrt{q}.$$
 It follows that for $ 0\leq k \leq n-1$,
 \begin{align}\label{5.1}\left|M_k(\chi)\right|\leq {r \choose k} \sqrt{q}^k\leq {n-1 \choose k} \sqrt{q}^k.
 \end{align}

Now, note that $\sum_{g\in\f[x], \text{monic}, \deg(g)\leq k} \chi(g)$ is the coefficient
     of $T^k$ in $L(\chi, T)/(1-T)$. Let now $\chi$ be a non-trivial character but trivial on
     $\f^*$. Then $L(\chi, T)$ has the trivial factor $(1-T)$ since
     $L(\chi, 1)=0$. This means that $L(\chi, T)/(1-T)$ is a polynomial of degree $n-2$ \cite{Wan}.
Then by (\ref{5.1}) one has
 $$\left| \sum_{g\in\f[x], \text{monic}, \deg(g)\leq k} \chi(g)\right|\leq  {n-2 \choose k} \sqrt{q}^k.$$

%
%

 \end{proof}

\section{Proof of Lemma \ref{lem1.1} }
In this section we will prove Lemma \ref{lem1.1}.
 For convenicnec we state it again as the following independent theorem.
 \begin{thm}\label{thm6.1}
 Let $f(x) \in \f [x]$ be a monic polynomial of degree $d=k+m\leq q-1$. Let $M(f, r)$  denote
the number of pairs $(D_r, g(x))$ with $D_r$ being a $r$-subset in $\f$ and $g(x)\in \f[x]$ of degree at most $k-1$ satisfying
$$ (f(x)+g(x))|_{D_r} \equiv 0.$$ Then for $k+1\leq r\leq d$,  we have
 \begin{align*}
\left| {M(f, r)}-{{q \choose r}q^{k-r}}\right| \leq  {\frac {q}p+m\sqrt{q}+r \choose r}{m-1 \choose d-r} \sqrt{q}^{d-r}.
 \end{align*}
\end{thm}

We first establish two lemmas which
allows us to compute $M(f,r)$ through the method of character sums defined over a residue polynomial ring.

 For $k\geq 0$, let $P_k$ denote the set
of all polynomials $h(x) \in \f [x]$ of degree at most $k$ with $h(0)=1$.
Let $N_2$ be defined by
\begin{align*}
N_2=&\#\{((x_1,  \dots, x_r), h) \in \f^r \times P_{d-r} \big | (1-x_1x)\cdots(1-x_r x)h(x)\equiv f(x)(\bmod \ x^{m+1})\},
 \end{align*}
 where we require that the $x_i$'s are distinct.
 Let $\chi$ be a character from $(\f[x]/(x^{m+1}))^*$ to $\mathbb{C}^*$.
 We extend this definition to $\f[x]/(x^{m+1})$ by defining $\chi(g)=0$ for $(g, x^{m+1})\ne 1$.
 Let $G$ denote the group of all characters $\chi$ such that $\chi(\f^*)=1$. This is an abelian group of order $|G|=q^m$.
For any real number $x$ and a positive integer $r$, define $(x)_r=x(x-1)\cdots(x-r+1)$ and let $(x)_0=1$.

 \begin{lem}\label{lem3.5} Assume that $f(0)=1$.
Then       \begin{align*}
\left| N_2-(q)_r q^{k-r} \right| \leq  (\frac {q}p+m\sqrt{q}+r-1)_r{m-1 \choose d-r} \sqrt{q}^{d-r}.
 \end{align*}
 \end{lem}

 \begin{proof}
 \begin{align*}
N_2&=\frac 1 {q^m}\sum_{(x_1,  \dots, x_r)\in \f^r, x_i \ne x_j}\sum_{h\in P_{d-r}}\sum_{\chi\in G }\chi( {(1-x_1x)\cdots(1-x_r x)h(x)}/{f(x)})\\
&=(q)_r{q^{k-r}}+\frac {1} {q^m}\sum_{1\ne \chi\in G }\chi^{-1}(f(x))W(\chi),
 \end{align*}
where $$W(\chi)=\sum_{(x_1,  \dots, x_r)\in \f^r, x_i \ne x_j}\sum_{h\in P_{d-r}}\chi( {(1-x_1x)\cdots(1-x_r x)h(x)}).$$

For each character $\chi$, the function $\chi((1-x_1x)\cdots(1-x_r x))$ is clearly symmetric in the $x_i$'s.
Recall that for a permutation $\tau=(i_1i_2\cdots i_{a_1})
  (j_1j_2\cdots j_{a_2})\cdots(l_1l_2\cdots l_{a_s})$  in the symmetric group $S_r$ with $1\leq a_i, 1 \leq i\leq s$,
  the subset $X_{\tau}$ of $X= \f^r$ is defined as
  \begin{align}
     X_{\tau}=\left\{
(x_1,\dots,x_r)\in \f^r,
 x_{i_1}=\cdots=x_{i_{a_1}},\ldots, x_{l_1}=\cdots=x_{l_{a_s}}
 \right\}.
\end{align}
Then a complex function $F_{\tau}(\chi)$ is defined:
$$F_{\tau}(\chi)=\sum_{(x_1,\ldots, x_r) \in X_{\tau} } \sum_{h\in P_{d-r}}\chi( {(1-x_1x)\cdots(1-x_r x)h(x)}).$$
Thus by the sieving formula (\ref{1.5}), one has
  \begin{align*}
N_2&={(q)_r q^{k-r}} +\frac {1} {q^m}\sum_{1\ne \chi\in G }\chi^{-1}(f(x))\sum_{\tau\in S_r}{\sign(\tau)F_{\tau}(\chi)}.
 \end{align*}
  Thus it suffices to estimate $F_{\tau}(\chi)$ for non-trivial $\chi$, where
$$
   F_{\tau}(\chi)=\left(\sum_{(x_1,\ldots, x_r) \in X_{\tau}} \prod_{i=1}^r\chi(1-x_ix)\right)\cdot \left(\sum_{h\in P_{d-r}}\chi(h(x))\right).
$$
We first estimate the second factor. Since $\chi$ is non-trivial, $\chi(\f^*)=1$ and $\chi(x)=0$, by Lemma \ref{weil}
we deduce
$$|\sum_{h\in P_{d-r}}\chi(h(x))| = |\sum_{h\in\f[x], \text{monic}, \deg(h)\leq d-r} \chi(h(x))| \leq {m-1\choose d-r}\sqrt{q}^{d-r}.$$
To estimate the first factor, we suppose $\tau$ is of type $(c_1,c_2,\ldots,c_{r})$, where $c_i$ is the number of $i$-cycles in $\tau$ for $1 \leq i\leq r$. Then the first
factor of $F_{\tau}(\chi)$ is
   \begin{align*}
G_{\tau}(\chi)&=\sum_{(x_1,\ldots, x_r) \in X_{\tau}} \prod_{i=1}^r\chi(1-x_ix)\\
&=(\sum_{a\in{{\bf{F}}_q}}\chi(1+ax))^{c_1}
(\sum_{a\in{{\bf{F}}_q}}\chi^2(1+ax))^{c_2} \cdots
(\sum_{a\in{{\bf{F}}_q}}\chi^{r}(1+ax))^{c_{r}}\\
&=\prod_{i=1}^{r}(\sum_{a\in{{\bf{F}}_q}}\chi^i(1+ax))^{c_i}.
 \end{align*}
Define $ m_i(\chi)=1$ if $\chi^i=1$ and  $ m_i(\chi)=0$ if $\chi^i\neq1$.
By the Weil bound (see \cite{Wan} Theorem 2.1) we deduce that
$$|G_{\tau}| \leq  q^{\sum_{i=1}^{r}c_im_i(\chi)}(m\sqrt{q})^{\sum_{i=1}^{r}
c_i(1-m_i(\chi))}.$$
Since $X=\f^r$ is symmetric, by (\ref{7}) we have
   \begin{align*}
  N_2-{(q)_r q^{k-r}}&=\frac {1} {q^m}\sum_{1\ne \chi\in G }\chi^{-1}(f(x))\sum_{\tau\in S_r}{\sign(\tau)F_{\tau}(\chi)}\\
  &=\frac {1} {q^m}\sum_{1\ne \chi\in G }\chi^{-1}(f(x))\sum_{\tau\in C_{r}}\sign(\tau)  C(\tau) F_{\tau}(\chi)\\
&=\frac{1}{q^m} \sum_{\chi^d \neq 1, \forall 2\leq d\leq r}\chi^{-1}(f(x))
\sum_{\tau\in C_{r}}\sign(\tau)C(\tau) F_{\tau}(\chi)\\
+\frac{1}{q^m}&\sum_{\chi\neq 1, \chi^d=1, \text{ for some}\  2\leq d\leq r
}\chi^{-1}(f(x)) \sum_{\tau\in C_{r}}\sign(\tau)C(\tau)
F_{\tau}(\chi).
 \end{align*}
 Let $S=\#\{\chi \in G\ |\  \chi\ne1, \chi^d = 1~ \text{for some}\  2\leq d\leq r \}$.
 The last two terms were estimated by a combinatorial counting argument (see \cite{LW3} page 2361).
 We thus obtain
 $$\left|  N_2-{(q)_r q^{k-r}}\right|\leq  w(S){m-1 \choose d-r} \sqrt{q}^{d-r},$$
 where $$w(S)=\left(\frac{q^m-S}{q^m}((m-1)\sqrt{q}+r-1)_{r}
+\frac {S} {q^m}\cdot(\frac {q}p+(m-1)\sqrt{q}+r-1)_{r}\right).$$
If $S$ is 0, we have the stronger estimate
   $$\left|  N_2-{(q)_r q^{k-r}}\right|\leq ((m-1)\sqrt{q}+r-1)_{r}{m-1 \choose d-r} \sqrt{q}^{d-r}. $$
 In general, we have the weaker estimate
 $$\left|  N_2-{(q)_r q^{k-r}}\right|\leq (\frac {q}p+(m-1)\sqrt{q}+r-1)_{r}{m-1 \choose d-r} \sqrt{q}^{d-r}. $$ \end{proof}

Similarly, if we consider the counting problem in $\f^*$, then we will
have a slightly different formula.
 \begin{lem}\label{lem3.6}  Let $N_2^*$ be defined by
\begin{align*}
N_2^*=&\#\{((x_1,  \dots, x_r), h)\in {\f^*}^r\times P_{d-r} \big | (1-x_1x)\cdots(1-x_r x)h(x)\equiv f(x)(\bmod \ x^{m+1})\},
 \end{align*}
 where we require that the $x_i$'s are distinct.
 Then for $f(0)=1$, we have      \begin{align*}
\left| N_2^*-(q-1)_rq^{k-r} \right| \leq  ((q-1)/p+m\sqrt{q}+r-1)_r{m-1 \choose d-r} \sqrt{q}^{d-r}.
 \end{align*}
 \end{lem}

Now, we assume that $f(x)\in \f[x]$ is a monic polynomial of degree $d$.
Suppose the top $s$ coefficients of $f$ are $\alpha=(a_{d-1}, \ldots, a_{k})$, i.e., $$f^{\alpha}(x)=x^d+a_{d-1}
x^{d-1}+\cdots+a_{k}x^{k}+\cdots.$$
For integer $k\geq 0$, let $\f[x]_k$ denote the set of polynomials $g\in \f[x]$ of degree at most $k$.

\begin{proof}[Proof of Theorem \ref{thm6.1}]
Note that $M(f,r)$ equals the number of
pairs $(D_r, g(x))$, where $D_r=(x_1,\ldots, x_r)$ is an $r$-subset of $\f$ and $g \in \f[x]_{k-1}$
such that there is a unique monic $w(x) \in \f$ of degree $d-r$ satisfying
 \begin{align} x^d+a_{d-1}x^{d-1}+\cdots+a_{k}x^{k}+g(x)=(x-x_1)\cdots(x-x_r)w(x). \label{6.2}
  \end{align}
Clearly $M(f,r)=N_r^{\alpha, 1}(d, m)+N_r^{\alpha, 2}(d, m)$, where $N_r^{\alpha, 1}(d,m)$
 equals the number of such pairs  $(D_r, g(x))$ with $D_r \subseteq \f^*$
 and   $N_r^{\alpha, 2}$ equals the number of such pairs $(D_r, g(x))$ with $D_r$ containing $0$.

Suppose $x_1=0$, by dividing $x$ on both sides of (\ref{6.2}), it is easy to check $N_r^{\alpha, 2}(d, m)=N_{r-1}^{\alpha, 1}(d-1, m)$.  It then suffices to count  $N_r^{\alpha, 1}(d, m)$.

Since now we have $x_i\in \f^*$, Substitute $x$ by $1/x$ one has
\[ \frac 1 {x^d}+a_{d-1}\frac 1 {x^{d-1}}+\cdots+a_{k}\frac 1 {x^{k}}+g(\frac 1 x)=(\frac 1 x-x_1)(\frac 1 x -x_2)\cdots(\frac 1 x-x_r)w(\frac 1x).\]
Multiplying $x^d$ on both sides we then have
$$ 1+a_{d-1}x+\cdots+a_{k}x^s+x^d g(\frac 1 x)=(1-x_1 x)(1-x_2 x)\cdots(1-x_r x)x^{d-r}w(\frac 1x ).$$
Note that $h(x)=x^{d-r}w(\frac 1x )$ is a polynomial of degree $\leq d-r$,  $x^{d}g(\frac 1x )$ is a polynomial divisible by $x^{m+1}$ and degree bounded by $d$.
It suffices to count the number of pairs $(D_r, h(x))$, where $D_r=(x_1,\ldots, x_r)$  is an $r$ subset of $\f^*$ and $h(x)\in \f [x] $ of degree $\leq d-r$ such that
$$1+a_{d-1}x+\cdots+a_{k}x^s\equiv(1-x_1 x)(1-x_2 x)\cdots(1-x_r x)h(x) (\bmod \ x^{m+1}).$$
Thus, if we let $N_2^*$ be defined as in  Lemma \ref{lem3.6}, then
\begin{align*}
N_r^{\alpha, 1}(d, m)=\frac {1}{r!}N_2^*.
\end{align*}
It follows that
 \begin{align*}
\left| N_r^{\alpha,1}(d, m)-{q-1 \choose r}q^{k-r} \right| \leq  {\frac {q-1}p+m\sqrt{q}+r-1 \choose r}{m-1 \choose d-r} \sqrt{q}^{d-r}.
 \end{align*}
 Similarly, by the estimate in Lemma \ref{lem3.5} one has
 \begin{align*}
\left| N_r^{\alpha,2}(d, m)-{q-1 \choose r-1}q^{k-r} \right| \leq  {\frac {q-1}p+m\sqrt{q}+r-2 \choose r-1}{m-1 \choose d-r} \sqrt{q}^{d-r}.
 \end{align*}
Finally we conclude
  \begin{align*}
&\left| M(f,r)-\left({q-1 \choose r}+{q-1 \choose r-1}\right)q^{k-r} \right| \\
&\leq   {\frac {q-1}p+m\sqrt{q}+r-1 \choose r}{m-1 \choose d-r} \sqrt{q}^{d-r}+  {\frac {q-1}p+m\sqrt{q}+r-2 \choose r-1}{m-1 \choose d-r} \sqrt{q}^{d-r}\\
&\leq {\frac {q-1}p+m\sqrt{q}+r \choose r}{m-1 \choose d-r} \sqrt{q}^{d-r}.
 \end{align*}
 \end{proof}

{ \bf Remark:}  As we mentioned in the first section,  our main results for standard Reed-Solomon codes extend to primitive Reed-Solomon codes with minor modification. In fact, in this case, things are simpler. What we need to count is exactly $\frac 1 {r!}N_2^*$, which is given in Lemma \ref{lem3.6}.

{\bf Acknowledgements.} The authors wish to thank the anonymous
 referees for their constructive and valuable suggestions.

 \bibliographystyle{plain}
\bibliography{Li-Wan7-CountingPolynomials}

\end{document}